\documentclass[11pt]{article}
\usepackage[T1]{fontenc}
\usepackage{lmodern}
\usepackage{amsmath, amssymb, amsfonts}
\usepackage{epsfig}

\def\be{\begin{equation}}
\def\ee{\end{equation}}
\def\bea{\begin{eqnarray}}
\def\eea{\end{eqnarray}}
\def\bes{\begin{eqnarray*}}
\def\ees{\end{eqnarray*}}

\def\nn{\nonumber}
\def\<{\langle}
\def\>{\rangle}
\def\lb{\label}
\def\bs{\setminus}

\def\R{{\mathbb{R}}}
\def\C{{\mathbb{C}}}
\def\Z{{\mathbb{Z}}}
\def\N{{\mathbb{N}}}
\def\U{{\mathbf{U}}}
\def\Q{{\mathbb{Q}}}

\def\RP{{\mathbb{R}P^{n}}}

\def\aa{{\alpha}}
\def\bb{{\beta}}

\def\th{{\theta}}
\def\om{{\omega}}
\def\Om{{\Omega}}

\def\lm{{\lambda}}

\def\sg{{\sigma}}

\def\K{{\cal K}}
\def\P{{\cal P}}

\def\ker{{\rm ker}}

\def\rank{{\rm rank}}
\def\Sp{{\rm Sp}}

\def\dm{{\diamond}}

\title{Multiplicity of closed geodesics on odd-dimensional bumpy Finsler compact space form}
\author{Hui Liu $^{1}$,\thanks{Partially supported by NSFC (12371195) . E-mail: huiliu00031514@whu.edu.cn.}\quad 
Linxinzi Qi$^2$,\thanks{ E-mail: 2025102010002@whu.edu.cn}\\
School of Mathematics and Statistics, Wuhan University,
\\Wuhan 430072, Hubei, China\\}

\date{}
\begin{document}
\newtheorem{definition}{Definition}[section]
\newtheorem{theorem}{Theorem}[section]
\newtheorem{lemma}{Lemma}[section]
\newtheorem{corollary}{Corollary}[section]
\newtheorem{conjecture}{Conjecture}[section]
\newtheorem{example}{Example}[section]
\newtheorem{property}{Property}[section]
\newtheorem{proposition}{Proposition}[section]
\newtheorem{remark}{Remark}[section]

\newcommand{\qed}{\nolinebreak\hfill\rule{2mm}{2mm}
\par\medbreak}
\newcommand{\Proof}{\par\medbreak\it Proof: \rm}
\newcommand{\rem}{\par\medbreak\it Remark: \rm}
\newcommand{\defi}{\par\medbreak\it Definition : \rm}
\renewcommand{\thefootnote}{\arabic{footnote}}

\maketitle

\begin{abstract}
Let $M = S^{2n+1}/\Gamma$ be a manifold diffeomorphic to a compact space form, endowed with a bumpy Finsler metric $F$.
In this paper, we prove that there exist at least $n+3$ contractible minimal closed geodesics on $M$,
assuming that all contractible minimal closed geodesics have non-zero Morse indices. Here, a contractible closed geodesic is called minimal
if it is not an iteration of any other contractible closed geodesic.
In particular, this confirms the generalized Anosov-Katok conjecture for the bumpy quotient manifold $(S^3/\Gamma, F)$ under the non-zero Morse index assumption.
\end{abstract}

{\bf Key words}: Contractible minimal closed geodesics;
Morse theory; the enhanced common index jump theorem

{\bf AMS Subject Classification}: 53C22, 58E05, 58E10.

\renewcommand{\theequation}{\thesection.\arabic{equation}}
\makeatletter\@addtoreset{equation}{section}\makeatother

\setcounter{equation}{0}
\section{Introduction}
Let $M=S^{2n+1}/ \Gamma$, where $\Gamma$ is a finite group which acts freely and
isometrically on the $(2n+1)$-sphere and therefore $M$ is diffeomorphic to a compact space form
which is typically a non-simply connected manifold. In 2019, Liu \cite{Liu2019} established the optimal lower
bound estimate of the number of closed geodesics on $S^{2n+1}/\Gamma$ with a bumpy irreversible Finsler metric
under some curvature condition. In this paper, we study the multiplicity of closed geodesics on the quotient
manifold $S^{2n+1}/\Gamma$ equipped with a bumpy Finsler metric, without any curvature condition.

A closed curve on a Finsler manifold is a closed geodesic if it is
locally the shortest path connecting any two nearby points on this
curve. As usual, on any Finsler manifold
$(M, F)$, a closed geodesic $c:S^1=\R/\Z\to M$ is {\it prime}
if it is not a multiple covering (i.e., iteration) of any other closed geodesics.
 Here the $m$-th iteration $c^m$ of $c$ is defined
by $c^m(t)=c(mt)$. The inverse curve $c^{-1}$ of $c$ is defined by
$c^{-1}(t)=c(1-t)$ for $t\in \R$.  Note that unlike Riemannian manifolds,
the inverse curve $c^{-1}$ of a closed geodesic $c$
on an irreversible Finsler manifold need not be a geodesic.
We call two prime closed geodesics
$c$ and $d$ {\it distinct} if there is no $\th\in (0,1)$ such that
$c(t)=d(t+\th)$ for all $t\in\R$.
For a closed geodesic $c$ on $(M,\,F)$, denote by $P_c$
the linearized Poincar\'{e} map of $c$. Recall that a Finsler metric $F$ is {\it bumpy} if all the closed geodesics
on $(M, \,F)$ are non-degenerate, i.e., $1\notin \sigma(P_c)$ for any closed
geodesic $c$.

The study of the existence and multiplicity of closed geodesics is a classical problem in differential geometry and dynamical systems.
There is a celebrated conjecture which claims the existence of infinitely many
distinct closed geodesics on compact Riemannian manifolds. This conjecture has been proved in many cases: For compact Riemannian manifolds,
Gromoll and Meyer \cite{GM1969JDG} in 1969 proved that there exist infinitely many distinct closed geodesics if the free loop space has an unbounded Betti number sequence. 
Then for compact simply connected manifolds Vigué-Poirrier and Sullivan \cite{VS1976} further proved in 1976 that the Betti number sequence is unbounded if and 
only if $H^{*}(M,\Q)$ is generated by more than one element. Franks \cite{Franks1992} and Bangert \cite{Bangert1993} proved that 
there are always infinitely many distinct closed geodesics on every Riemannian sphere $S^2$ (cf. also Hingston \cite{Hingston1993}, Klingenberg \cite{Kli1995}).

However, in the Finsler case, this problem becomes completely different.
In 1973, Katok \cite{Katok1973} constructed remarkable examples of irreversible Finsler metrics on compact rank one symmetric spaces (CROSSs) with finitely many
prime closed geodesics. In particular, there are exactly $2[\frac{n+1}{2}]$ prime closed geodesics on $S^n$ with Katok's metric (cf. also Ziller \cite{Ziller1977}, \cite{Ziller1982}).

Based on it, Anosov \cite{Ano} conjectured in 1974 that the optimal lower bound for the number of prime closed geodesics on any Finsler sphere $S^n$ should be $2[\frac{n+1}{2}]$.
This conjecture was solved for $S^2$ by Bangert and Long \cite{BL2010} in 2004 (published in 2010). For more interesting results and recent progress regarding this conjecture,
we refer the reader to \cite{DL2010}, \cite{HiR}, \cite{LD2009}, \cite{Rad2007}, \cite{Wang2008}, \cite{Wang2012}.
A fundamental tool developed to address such multiplicity problems on Finsler manifolds is the index iteration theory for symplectic paths.
In particular, the enhanced common index jump theorem established by Duan, Long, and Wang \cite{DLW2} has played an important role in recent advancements.
This enhanced theorem provides precise estimates for the Maslov-type indices of the iterates of closed geodesics.
Using this tool, Duan and Qi \cite{DQ2025} recently solved the Anosov conjecture for the bumpy Finsler $S^3$ assuming all prime closed geodesics have positive Morse indices.
Very recently, Duan and Xie \cite{DX2026} proved the existence of $\frac{n+5}{2}$ prime closed geodesics on any bumpy Finsler sphere $(S^n, F)$ with an odd integer $n \ge 3$,
under the assumption that all geometrically distinct prime closed geodesics possess positive Morse indices.

For any $S^n/\Gamma$, it is a natural conjecture that there exist at least $2[\frac{n+1}{2}]$ distinct closed geodesics, which is called the generalized Anosov-Katok
conjecture. There are many active studies on this conjecture in recent years. For example, it is well known that the only non-trivial group which acts freely on $S^{2n}$ is $\Z_2$
and $S^{2n}/\Z_2$ is $\R P^{2n}$. In 2016, Taimanov \cite{Tai2016} proved the existence of at least two distinct non-contractible closed geodesics on $\R P^2$ endowed with a bumpy irreversible Finsler metric. 
In \cite{LX},  Liu and Xiao established
the resonance identity for the non-contractible closed geodesics on $\mathbb{R}P^n$, and
together with \cite{DLX2015} and \cite{Tai2016} proved the existence of at least two distinct
non-contractible closed geodesics on every bumpy $\mathbb{R}P^n$ with $n\geq2$.
Furthermore, Liu, Long and Xiao \cite{LLX}
proved that every bumpy Finsler compact space form $S^n/ \Gamma$ possesses two distinct
closed geodesics in each of its nontrivial classes.
In 2018, Ginzburg, Gurel and Macarini \cite{GGM} established this conjecture on bumpy Finsler $\mathbb{R}P^{2n}$ under some index conditions. So it is natural to ask about the multiplicity of closed geodesics on an odd-dimensional quotient manifold $S^{2n+1}/\Gamma$.

Motivated by Duan and Xie \cite{DX2026} and Liu \cite{Liu2019}, in this paper, we drop the positive curvature assumption in Liu \cite{Liu2019}
and obtain a multiplicity result for contractible minimal closed geodesics on the bumpy odd-dimensional Finsler quotient manifold $(S^{2n+1}/\Gamma,F)$, as stated in the following theorem. Note that a contractible closed
geodesic c is called minimal if it is not an iteration of any other contractible closed
geodesics.
\begin{theorem}\label{Thm1.1}
     Let $M=S^{2n+1}/\Gamma$ be endowed with a bumpy Finsler metric $F$.
    If all contractible minimal closed geodesics have non-zero Morse indices, then there exist $n+3$ contractible minimal closed geodesics.
\end{theorem}

In particular, this result confirms the generalized Anosov-Katok conjecture for the 3-dimensional bumpy Finsler quotient manifold $(S^3/\Gamma, F)$ under the non-zero Morse index assumption.  
\begin{remark}\label{rk1.1}
    It should be noted that our result does not follow directly from the work of Duan and Xie \cite{DX2026}. This is because two distinct closed geodesics on the covering sphere might project to the same closed geodesic on the underlying quotient manifold.
\end{remark}

Suppose the assumption of Theorem 1.1 fails, meaning that the space $(S^{2n+1}/\Gamma, F)$ admits at least one contractible minimal closed geodesic with zero Morse index.
In this context, we recall a classical result by Bangert and Klingenberg \cite{BK1983}: for a compact manifold $M$ with a finite fundamental group,
if there exists a closed geodesic $c$ whose iterates $c^m$ are local minima of the energy functional $E$ for infinitely many $m \in \mathbb{N}$,
then $M$ admits infinitely many prime closed geodesics.

Furthermore, note that when a contractible minimal closed geodesic $c$ is non-degenerate,
the fact that $c$ is not a local minimum of the energy functional $E$ implies that its Morse index satisfies $i(c) > 0$.
Combining this observation with the above result of Bangert and Klingenberg,
we strongly suspect the existence of infinitely many contractible minimal closed geodesics in the case when the assumption of Theorem 1.1 fails.

Motivated by these considerations, we propose the following conjecture.
\begin{conjecture}\label{conj1.1}
    Let $(S^{n}/\Gamma, F)$ be a Finsler compact space form of dimension $n$.
If there exist only finitely many contractible minimal closed geodesics on $(S^{n}/\Gamma,F)$,
then none of the closed geodesics (including their iterates) is a local minimum of the energy functional $E$.
\end{conjecture}
In this paper, let $\N$, $\N_0$,  $\Z$, and $\Q$ denote the sets of natural integers, non-negative
integers, integers, rational numbers respectively.
We also use  notations $\lceil a \rceil=\min\{k\in\Z\,|\,k\ge a\}$,
$[a]=\max\{k\in\Z\,|\,k\le a\}$, $\varphi(a)=\lceil a \rceil-[a]$ and $\{a\}=a-[a]$ for any $a\in\R$.
Throughout this paper, we use $\Q$ coefficients for all homological and cohomological modules.
\section{Critical point theory for closed geodesics }
Let $(M,F)$ be a compact Finsler manifold, the space
$\Lambda=\Lambda M$ of $H^1$-maps $\gamma:S^1\rightarrow M$ has a
natural structure of Riemannian Hilbert manifolds on which the
group $S^1=\R/\Z$ acts continuously by isometries. This action is defined by
$(s\cdot\gamma)(t)=\gamma(t+s)$ for all $\gamma\in\Lambda$ and $s,
t\in S^1$. For any $\gamma\in\Lambda$, the energy functional is
defined by
\be  E(\gamma)=\frac{1}{2}\int_{S^1}F(\gamma(t),\dot{\gamma}(t))^2dt.
\lb{2.1}\ee
It is $C^{1,1}$ and invariant under the $S^1$-action. The
critical points of $E$ of positive energies are precisely the closed geodesics
$\gamma:S^1\to M$. The index form of the functional $E$ is well
defined along any closed geodesic $c$ on $M$, which we denote by
$E''(c)$. As usual, we denote by $i(c)$ the Morse index of $E$ at
$c$ and by $\nu(c)=\dim\ker(P_c-I)$ its reduced nullity.
For a closed geodesic $c$ we set $ \Lambda(c)=\{\gamma\in\Lambda\mid E(\gamma)<E(c)\}$.

Recall that respectively the mean index $\hat{i}(c)$ and the $S^1$-critical modules of $c^m$ are defined by
\be \hat{i}(c)=\lim_{m\rightarrow\infty}\frac{i(c^m)}{m}, \quad \overline C_*(E,c^m)
   = H_*\left((\Lambda(c^m)\cup S^1\cdot c^m)/S^1,\Lambda(c^m)/S^1; \Q\right).\lb{2.2}\ee

We call a closed geodesic $c$ satisfying the isolation condition, if
the following holds:

{\bf (Iso)  For all $m\in\N$ the orbit $S^1\cdot c^m$ is an
isolated critical orbit of $E$. }

Note that if the number of prime closed geodesics on a Finsler manifold
is finite, then all the closed geodesics satisfy (Iso).

If $c$ has multiplicity $m$, then the subgroup $\Z_m=\{\frac{n}{m}\mid 0\leq n<m\}$
of $S^1$ acts on $\overline C_*(E,c)$. As studied in p.59 of \cite{Rad1992},
for all $m\in\N$, let
$H_{\ast}(X,A)^{\pm\Z_m}
   = \{[\xi]\in H_{\ast}(X,A)\,|\,T_{\ast}[\xi]=\pm [\xi]\}$,
where $T$ is a generator of the $\Z_m$-action.

In the following, we let $M=S^{2n+1}/ \Gamma$ and $e$ be the identity in $\pi_1(M)$, where
$\Gamma$ acts freely and isometrically on the $(2n+1)$-sphere and therefore
$M$ is diffeomorphic to a compact space form. Then the free loop space $\Lambda M$ possesses
the natural decomposition
\bea \Lambda M=\bigsqcup_{[g]\in\operatorname{Conj}(\pi_1(M))}\Lambda_{[g]}M,\nn\eea
where $\Lambda_{[g]}M$ is the component corresponding to the conjugacy
class $[g]$. We set $\Lambda_{e}(c) = \{\gamma\in \Lambda_{e}M\mid E(\gamma)<E(c)\}$.
Note that for a contractible minimal closed geodesic $c$, $c^m\in\Lambda_e M$
for every $m\in\N$.

Every nonconstant closed geodesic has a unique prime root up to an
$S^1$-shift. Write a contractible minimal closed geodesic as $c=d^p$,
where $d$ is prime. If $g\in\pi_1(M)$ is any representative of the
free homotopy class of $d$, then
\be p=\operatorname{ord}(g)=\operatorname{mult}(c). \lb{root-order}\ee
Equation (\ref{root-order}) follows because contractibility gives
$g^p=e$, while $\operatorname{ord}(g)<p$ would express $c$ as a
nontrivial iterate of another contractible closed geodesic. 
Thus $d$ may be non-contractible, while $c=d^p$ is precisely its first
contractible iterate.

Now we restrict the energy functional $E$ on the contractible component $\Lambda_e M$ and study the Morse theory on $\Lambda_e M$.
Thus we define the $S^1$-critical modules of $c^m$ for $E|_{\Lambda_e M}$ as
\be \overline C_*(E,c^m; [e])
   = H_*\left((\Lambda_e(c^m)\cup S^1\cdot c^m)/S^1,\Lambda_e(c^m)/S^1; \Q\right).\ee

In 2019, Liu \cite{Liu2019} developed the critical group theory for the contractible component of the quotient manifold $S^{2n+1}/\Gamma$ and computed the $S^1$-equivariant Betti number sequence of $ (\Lambda_eM, \Lambda^0M)$ via Katok's famous metrics on $S^n$, where $\Lambda^0 M  =\{{\rm constant\;point\;curves\;in\;}M\}\cong M$. So we have the following lemmas and theorems.
\begin{definition}(cf. Definition 5.1 of \cite{Liu2019})
A contractible minimal closed geodesic $c$ is \emph{homologically
visible} if $\overline C_{\ell}(E,c;[e])\ne0$ for some $\ell\in\Z$.
Otherwise it is \emph{homologically invisible}.
\end{definition}

\begin{proposition}(cf. Satz 6.11 of \cite{Rad1992} and Proposition 2.1 of \cite{Liu2019})
\label{Rad1992'}
Let $c=d^p$ be a contractible minimal closed geodesic on
$M=S^{2n+1}/\Gamma$, where $d$ is prime, and assume that $c$
satisfies \textup{(Iso)}. For every $m\in\N$, the full isotropy group is
\be
G_m:=\operatorname{Stab}_{S^1}(c^m)
=\operatorname{Stab}_{S^1}(d^{pm})=\Z_{pm}. \lb{full-isotropy}
\ee
There exist a $G_m$-invariant local negative disk at $c^m$ and a
$G_m$-invariant local characteristic manifold $N_{c^m}$. Denote the
punctured negative disk by $U_{c^m}^-$ and put
$N_{c^m}^-=N_{c^m}\cap\Lambda_e(c^m)$. Then
$\nu(c^m)=\dim N_{c^m}$ and, for every $\ell\in\Z$,
\begin{align}
\overline C_{\ell}(E,c^m;[e])
\cong\Big(&H_{i(c^m)}(U_{c^m}^-\cup\{c^m\},U_{c^m}^-;\Q)\nonumber\\
&\otimes H_{\ell-i(c^m)}
(N_{c^m}^-\cup\{c^m\},N_{c^m}^-;\Q)\Big)^{G_m}. \lb{local-splitting}
\end{align}

Let $T_m$ be the generator of $G_m$ induced by the time shift
$t\mapsto t+1/(pm)$. Its action on the orientation line of the
negative disk is multiplication by
\be
\epsilon(c^m)=(-1)^{i(d^{pm})-i(d)}=(-1)^{i(c^m)-i(d)}. \lb{orientation-character}
\ee
Consequently,
\be
\overline C_{\ell}(E,c^m;[e])\cong
H_{\ell-i(c^m)}(N_{c^m}^-\cup\{c^m\},N_{c^m}^-;\Q)
 ^{\epsilon(c^m)G_m}, \lb{local-characteristic}
\ee
where the superscript denotes the $\epsilon(c^m)$-eigenspace of
$(T_m)_*$, that is,
$V^{\epsilon(c^m)G_m}=\{\xi\in V\mid (T_m)_*\xi=\epsilon(c^m)\xi\}$.
In particular, if $\nu(c^m)=0$, then
\be
\overline C_{\ell}(E,c^m;[e])=
\begin{cases}
\Q,&\ell=i(c^m)\ \text{and}\ i(d^{pm})-i(d)\in2\Z,\\
0,&\text{otherwise}.
\end{cases} \lb{bumpy-full-isotropy}
\ee

If $F$ is bumpy, then $c$ is homologically visible if and only if
$i(c)-i(d)\in2\Z$. For a visible $c$,
\be
\overline C_{\ell}(E,c^m;[e])=
\begin{cases}
\Q,&\ell=i(c^m)\ \text{and}\ i(c^m)-i(c)\in2\Z,\\
0,&\text{otherwise},
\end{cases} \lb{visible-tower}
\ee
whereas for an invisible $c$ one has
\be \overline C_{\ell}(E,c^m;[e])=0,
\qquad \ell\in\Z,\quad m\in\N. \lb{invisible-tower}\ee
\end{proposition}

{\it Proof.}
Since $d$ is prime, $c^m=d^{pm}$ has isotropy group $\Z_{pm}$.

The equivariant Gromoll--Meyer splitting lemma, applied with this full
isotropy group, gives (\ref{local-splitting}). The action on the
negative orientation line is the standard index-orientation character
(\ref{orientation-character}); see Satz 6.11 of \cite{Rad1992} and
Lemma 2.1 of \cite{LLX}. The free action on $S^{2n+1}$ is
orientation preserving, so $M$ and the normal bundle of $d$ are
orientable. This proves (\ref{local-characteristic}) and
(\ref{bumpy-full-isotropy}).

For the last assertions, use the parity relation
$i(d^{k+2})-i(d^k)\in2\Z$. If $i(d^p)-i(d)$ is even, the condition in
(\ref{bumpy-full-isotropy}) is equivalent to
$i(c^m)-i(c)\in2\Z$. If $i(d^p)-i(d)$ is odd, the corresponding
orientation character of $\Z_p$ is nontrivial, hence $p$ is even.
Then $i(d^{pm})-i(d)$ is odd for every $m\ge1$, and bumpiness implies
(\ref{invisible-tower}). \hfill $\square$

\begin{proposition}(cf. Proposition 4.1 of \cite{Liu2019})
    For $M=S^{2n+1}/ \Gamma$, the $S^1$-equivariant Betti numbers
of $(\Lambda_e M, \Lambda^0 M)$ are given by
\bea \beta_j\equiv\rank H_j(\Lambda_e M/S^1,\Lambda^0 M/S^1;\Q)  =\left\{\begin{array}{ll}
    2,&\quad {\it if}\quad j\in \K\equiv \{2nk\,|\,2\le k\in\N\},  \\
    1,&\quad {\it if}\quad j\in \{2n+2k\,|\,k\in\N_0\}\bs\K,  \\
    0 &\quad {\it otherwise}. \\
\end{array}
\right. \lb{ksfzbts}\eea
and the average $S^1$-equivariant Betti number of $(\Lambda_e M, \Lambda^0 M)$ satisfies
\be  \overline{B}(\Lambda_e M, \Lambda^0 M;\Q)\equiv\lim_{L\to+\infty}\frac{1}{L}
\sum_{k=0}^{L}(-1)^k\beta_k= \frac{n+1}{2n}.   \lb{aB.1}\ee
\end{proposition}

For a contractible minimal closed geodesic $c_j=d_j^{p_j}$, define
\be
\epsilon_{j,m}=(-1)^{i(d_j^{p_jm})-i(d_j)},\qquad
k_{\ell}^{\epsilon_{j,m}}(c_j^m)
=\dim H_{\ell}(N_{c_j^m}^-\cup\{c_j^m\},N_{c_j^m}^-;\Q)
 ^{\epsilon_{j,m}\Z_{p_jm}}. \lb{local-type-number}
\ee
Put
\be
a_j(m)=\sum_{\ell=0}^{4n}(-1)^{\ell+i(c_j^m)}
k_{\ell}^{\epsilon_{j,m}}(c_j^m),\qquad
\widehat\chi(c_j)=\lim_{L\to\infty}\frac1L\sum_{m=1}^{L}a_j(m).
\lb{mean-euler-number}
\ee
The analytic-period theorem in \cite{DL2010} and \cite{LD2009},  applied to the prime root $d_j$ shows
that the full-isotropy local data along the subsequence
$d_j^{p_jm}$ are periodic; hence the limit exists.

\begin{theorem}(cf. Theorem 1.1 of \cite{Liu2019})
    Let $M=S^{2n+1}/ \Gamma$ and $e$ be the identity in $\pi_1(M)$.
Suppose the Finsler manifold $(M,F)$ possesses only finitely many contractible minimal closed geodesics,
among which we denote the distinct homologically visible contractible minimal
closed geodesics by $c_1, \ldots, c_r$  for some integer $r>0$. Then we have
\bea  \sum_{j=1}^{r}\frac{\widehat{\chi}(c_j)}{\hat{i}(c_j)} = \overline{B}(\Lambda_e M, \Lambda^0 M;\Q) =
    \frac{n+1}{2n},     \lb{reident1}\eea
where $\Lambda_{e} M$ is the contractible loop space of $M$,
$\Lambda^0 M  =\{{\rm constant\;point\;curves\;in\;}M\}\cong M$,
and the local quantities are defined by (\ref{local-type-number}) and
(\ref{mean-euler-number}).

In particular, if the Finsler metric $F$ on $M=S^{2n+1}/ \Gamma$ is bumpy,
then $a_j(m+2)=a_j(m)$ and
\be
\gamma_{c_j}:=\widehat\chi(c_j)=\frac{a_j(1)+a_j(2)}{2}. \lb{gamma-definition}
\ee
Thus (\ref{reident1}) has the following simple form:
\bea  \sum_{j=1}^{r}\frac{\gamma_{c_j}}{\hat{i}(c_{j})}
=\frac{n+1}{2n}.  \lb{breident1}\eea
\end{theorem}

The resonance identity above is used as a quoted result; its
Morse-series proof in \cite{Liu2019} applies with the full-isotropy
critical modules of Proposition~\ref{Rad1992'}. Invisible iteration
towers make no contribution by (\ref{invisible-tower}).
\begin{theorem}(cf. Sections 4 and 5 of \cite{Liu2019})
Suppose that there exist only finitely many
contractible minimal closed geodesics $\{c_j\}_{1\le j\le q}$ on a compact Finsler space form $(S^{2n+1}/\Gamma,F)$. Define
\be M_p=\sum_{1\le j\le q,\ m\ge 1}\dim \overline C_p(E,c_j^m;[e]),\qquad\forall\ p\in\Z. \lb{2.9}\ee
Then the following Morse inequality holds for every $p\in \N_0$,
\bea M_p &\ge& \beta_p,\lb{2.10}\\
M_p - M_{p-1} + \cdots +(-1)^{p}M_0
&\ge& \beta_p - \beta_{p-1}+ \cdots + (-1)^{p}\beta_0. 
\lb{2.11}\eea
\end{theorem}

\section{The enhanced common index jump theorem of symplectic paths}

In \cite{lo1999} of 1999, Y. Long established the basic normal form decomposition of symplectic matrices.
Based on this result he further established the precise iteration formulae of indices of
symplectic paths in \cite{lo2000} of 2000. Note that this index iteration formulae works for Morse indices
of iterated closed geodesics (cf. \cite{LL2002}, \cite{Liu2005} and Chap. 12 of \cite{lo2002}).
Since $M=S^{2n+1}/\Gamma$ is orientable, the normal bundle of every
closed geodesic on $M$ is orientable. Hence, by Theorem 1.1 of \cite{Liu2005},
the initial Morse index of a closed geodesic $c$ on a Finsler quotient manifold $S^{2n+1}/\Gamma$ coincides with the index of a
corresponding symplectic path.

As in \cite{lo2000}, denote by
\bea
N_1(\lm, b) &=& \left(
  \begin{array}{ll}\lm & b\cr
                                0 & \lm\\ \end{array}
  \right), \qquad {\rm for\;}\lm=\pm 1, \; b\in\R, \lb{3.1}\\
D(\lm) &=& \left(
  \begin{array}{ll}\lm & 0\cr
                      0 & \lm^{-1}\\ \end{array}
  \right), \qquad {\rm for\;}\lm\in\R\bs\{0, \pm 1\}, \lb{3.2}\\
R(\th) &=& \left(
  \begin{array}{ll}\cos\th & -\sin\th \cr
                           \sin\th & \cos\th\\ \end{array}
  \right), \qquad {\rm for\;}\th\in (0,\pi)\cup (\pi,2\pi), \lb{3.3}\eea
\bea N_2(e^{\th\sqrt{-1}}, B) &=& \left(
  \begin{array}{ll} R(\th) & B \cr
                  0 & R(\th)\\ \end{array}
  \right), \qquad {\rm for\;}\th\in (0,\pi)\cup (\pi,2\pi)\;\; {\rm and}\; \nn\\
        && \qquad B=\left(
  \begin{array}{ll}b_1 & b_2\cr
                                  b_3 & b_4\\ \end{array}
  \right)\; {\rm with}\; b_j\in\R, \;\;
                                         {\rm and}\;\; b_2\not= b_3. \lb{3.4}\eea
Here $N_2(e^{\th\sqrt{-1}}, B)$ is non-trivial if $(b_2-b_3)\sin\theta<0$, and trivial
if $(b_2-b_3)\sin\theta>0$.

As in \cite{lo2000}, the $\diamond$-sum (direct sum) of any two real matrices is defined by
\bea \left(
  \begin{array}{ll}A_1 & B_1\cr C_1 & D_1\\ \end{array}
  \right)_{2i\times 2i}\diamond
      \left(
  \begin{array}{ll}A_2 & B_2\cr C_2 & D_2\\ \end{array}
  \right)_{2j\times 2j}
=\left(
  \begin{array}{llll}A_1 & 0 & B_1 & 0 \cr
                                   0 & A_2 & 0& B_2\cr
                                   C_1 & 0 & D_1 & 0 \cr
                                   0 & C_2 & 0 & D_2\\ \end{array}
  \right).\lb{3.5}\eea

For every $M\in\Sp(2n)$, the homotopy set $\Omega(M)$ of $M$ in $\Sp(2n)$ is defined by
$$ \Om(M)=\{N\in\Sp(2n)\,|\,\sg(N)\cap\U=\sg(M)\cap\U\;\mbox{and}
                    \;\nu_{\om}(N)=\nu_{\om}(M)\, \forall\om\in\U \}, $$
where $\sg(M)$ denotes the spectrum of $M$, $\U=\{z\in \C:|z|=1 \}$, 
$\nu_{\om}(M)\equiv\dim_{\C}\ker_{\C}(M-\om I)$ for $\om\in\U$.
We denote by $\Om^0(M)$ the path connected component of $\Om(M)$ containing $M$.
\begin{lemma}(cf. \cite{lo2000}, Definition 9.1.4, Lemma 9.1.5 and List 9.1.12 of \cite{lo2002})
For $M\in\Sp(2n)$ and $\om\in\U$, the splitting number $S_M^\pm(\om)$ satisfies
\bea
S_M^{\pm}(\om) &=& 0, \qquad {\it if}\;\;\om\not\in\sg(M).  \lb{3.6}\\
S_{N_1(1,a)}^+(1) &=& \left\{
  \begin{array}{ll}1, &\quad {\rm if}\;\; a\ge 0, \cr
0, &\quad {\rm if}\;\; a< 0.  \end{array}\right.
  \lb{3.7}\eea

For any $M_i\in\Sp(2n_i)$ with $i=0$ and $1$, there holds
\be S^{\pm}_{M_0\diamond M_1}(\om) = S^{\pm}_{M_0}(\om) + S^{\pm}_{M_1}(\om),
    \qquad \forall\;\om\in\U. \lb{3.8}\ee\end{lemma}

For every $\gamma\in\mathcal{P}_\tau(2n)\equiv\{\gamma\in C([0,\tau],\Sp(2n))\ |\ \gamma(0)=I_{2n}\}$, we extend
$\gamma(t)$ to $t\in [0,m\tau]$ for every $m\in\N$ by
\bea \gamma^m(t)=\gamma(t-j\tau)\gamma(\tau)^j \qquad \forall\;j\tau\le t\le (j+1)\tau \;\;
               {\rm and}\;\;j=0, 1, \ldots, m-1, \lb{3.9}\eea
as in P.114 of \cite{lo1999}. As in \cite{LoZ} and \cite{lo2002}, we denote the Maslov-type indices of
$\gamma^m$ by $(i(\gamma,m),\nu(\gamma,m))$.

Then the following decomposition theorem and precise iteration formula for symplectic paths are proved in \cite{lo1999}
and \cite{lo2000}.
\begin{theorem}\label{jqddgs}(cf. Theorem 1.8.10, Lemma 2.3.5 and Theorem 8.3.1 of \cite{lo2002})
For every $M\in\Sp(2n)$, there
exists a continuous path $f\in\Om^0(M)$ such that $f(0)=M$ and
\bea f(1)
&=& N_1(1,1)^{\dm p_-}\,\dm\,I_{2p_0}\,\dm\,N_1(1,-1)^{\dm p_+}
  \dm\,N_1(-1,1)^{\dm q_-}\,\dm\,(-I_{2q_0})\,\dm\,N_1(-1,-1)^{\dm q_+} \nn\\
&&\dm\,N_2(e^{\aa_{1}\sqrt{-1}},A_{1})\,\dm\,\cdots\,\dm\,N_2(e^{\aa_{r_{\ast}}\sqrt{-1}},A_{r_{\ast}})
  \dm\,N_2(e^{\bb_{1}\sqrt{-1}},B_{1})\,\dm\,\cdots\,\dm\,N_2(e^{\bb_{r_{0}}\sqrt{-1}},B_{r_{0}})\nn\\
&&\dm\,R(\th_1)\,\dm\,\cdots\,\dm\,R(\th_r)\dm\,D(\pm 2)^{\dm h},\lb{3.10}\eea
where
$\frac{\th_{j}}{2\pi}\in(0,1)$ for $1\le j\le r$; $N_2(e^{\aa_{j}\sqrt{-1}},A_{j})$'s
are nontrivial and $N_2(e^{\bb_{j}\sqrt{-1}},B_{j})$'s are trivial, and non-negative integers
$p_-, p_0, p_+,q_-, q_0, q_+,r,r_\ast,r_0,h$ satisfy
\be p_- + p_0 + p_+ + q_- + q_0 + q_+ + r + 2r_{\ast} + 2r_0 + h = n. \nn\ee

Let $\gamma\in\P_{\tau}(2n)=\{\gamma\in C([0,\tau],\Sp(2n))\,|\,\gamma(0)=I\}$. Denote the basic normal form
decomposition of $M\equiv \gamma(\tau)$ by (\ref{3.10}). Then we have
\begin{align}\lb{3.11}
i(\gamma, m) = & \ m(i(\gamma, 1) + p_- + p_0 - r) + 2 \sum_{j=1}^{r} \lceil \frac{m\theta_j}{2\pi} \rceil - r - p_- - p_0 \nonumber \\
& - \frac{1 + (-1)^m}{2} (q_0 + q_+) + 2 \left( \sum_{j=1}^{r_*} \varphi \left( \frac{m\alpha_j}{2\pi} \right) - r_* \right).  \\
\nu(\gamma, m) = & \ \nu(\gamma, 1) + \frac{1 + (-1)^m}{2} (q_- + 2q_0 + q_+) + 2(r + r_* + r_0) \nonumber \\
& - 2 \left( \sum_{j=1}^{r} \varphi \left( \frac{m\theta_j}{2\pi} \right) + \sum_{j=1}^{r_*} \varphi \left( \frac{m\alpha_j}{2\pi} \right) + \sum_{j=1}^{r_0} \varphi \left( \frac{m\beta_j}{2\pi} \right) \right)  \\
\hat{i}(\gamma, 1) = & \ i(\gamma, 1) + p_- + p_0 - r + \sum_{j=1}^{r} \frac{\theta_j}{\pi}. 
\end{align}
\end{theorem} 

The common index jump theorem (cf. Theorem 4.3 of \cite{LoZ}) for symplectic paths established by Long
and Zhu in 2002 has become one of the main tools to study the multiplicity and stability problems of
closed solution orbits in Hamiltonian and symplectic dynamics. Recently, the following enhanced version
of it has been obtained by Duan, Long and Wang in \cite{DLW2}, which will play an important role in the
proofs.
\begin{theorem}\label{ECIJT}(\textit{cf. Theorem 3.5 of \cite{DLW2}}) (\textbf{The enhanced common index jump theorem for symplectic paths})
    Let $\{\gamma_k\}_{k=1}^q$, with $\gamma_k \in \mathcal{P}_{\tau_k}(2n)$ for each $k$, be a finite collection of symplectic paths.
    Let $M_k = \gamma_k(\tau_k)$. We extend $\gamma_k$ to $[0, +\infty)$ by (\ref{3.9}) inductively. Suppose
\begin{equation}
\hat{i}(\gamma_k, 1) > 0, \quad \forall k = 1, \dots, q. 
\end{equation}

\noindent Then for every integer $\bar{m} \in \mathbb{N}$, there exist infinitely many $(q + 1)$-tuples $(N, m_1, \dots, m_q) \in \mathbb{N}^{q+1}$ such that for all $1 \le k \le q$ and $1 \le m \le \bar{m}$, there holds
\begin{align}
\nu(\gamma_k, 2m_k - m) &= \nu(\gamma_k, 2m_k + m) = \nu(\gamma_k, m),\lb{3.15}  \\
i(\gamma_k, 2m_k + m) &= 2N + i(\gamma_k, m),\lb{3.16}  \\
i(\gamma_k, 2m_k - m) &= 2N - i(\gamma_k, m) - 2(S_{M_k}^+(1) + Q_k(m)),\lb{3.17}  \\
i(\gamma_k, 2m_k) &= 2N - (S_{M_k}^+(1) + C(M_k) - 2\Delta_k),\lb{3.18} 
\end{align}
where $C(M_k)=\sum_{0<\theta<2\pi}S^{-}_{M_k}(e^{\sqrt{-1}\theta})$, letting $\delta \in (0,\frac{1}{2})$ which satisfies
$\delta \max_{1\leq k \leq q} C(M_k)<\frac{1}{2}$, 
\begin{equation}\lb{3.19}
\Delta_k = \sum_{0 < \{m_k\theta/\pi\} < \delta} S_{M_k}^-(e^{\sqrt{-1}\theta}), \quad Q_k(m) = \sum_{\substack{e^{\sqrt{-1}\theta} \in \sigma(M_k), \\ \{\frac{m_k\theta}{\pi}\} = \{\frac{m\theta}{2\pi}\} = 0}} S_{M_k}^-(e^{\sqrt{-1}\theta}). 
\end{equation}
More precisely, by (4.10), (4.40) and (4.41) in \cite{lo2002}, we have
\begin{equation}\lb{3.20}
m_k = \left( \left\lfloor \frac{N}{\overline M\hat{i}(\gamma_k, 1)} \right\rfloor + \chi_k \right) \overline M, \quad 1 \le k \le q, 
\end{equation}
where $\chi_k = 0$ or $1$ for $1 \le k \le q$ and $\frac{\overline M\theta}{\pi} \in \mathbb{Z}$ whenever $e^{\sqrt{-1}\theta} \in \sigma(M_k)$ and $\frac{\theta}{\pi} \in \mathbb{Q}$ for some $1 \le k \le q$. Furthermore, given $M_0 \in \mathbb{N}$, by (iv) of Remark 3.6 of \cite{DLW2}, we may further require $M_0 | N$, and by (4.20) in Theorem 4.1 of \cite{LoZ}, for any $\epsilon > 0$, we can choose $N$ and $\{\chi_k\}_{1 \le k \le q}$ such that
\begin{equation}\lb{3.21}
\left| \left\{ \frac{N}{\overline M\hat{i}(\gamma_k, 1)} \right\} - \chi_k \right| < \epsilon, \quad 1 \le k \le q. 
\end{equation}
In particular, if $\nu(\gamma_k, m)=0$ for all $1\leq k\leq q$ and $m\in\N$, by (\ref{3.6}) and (\ref{3.19}) we have $S^+_{M_k}(1)=0$ and $Q_k(m)=0$, $\forall\ m\ge 1$,
then (\ref{3.16})-(\ref{3.18}) have the following simple forms
\begin{align}
    i(\gamma_k,2m_k+m)&=2N+i(\gamma_k,m),   \\                      
i(\gamma_k,2m_k-m)&=2N-i(\gamma_k,m),  \\
i(\gamma_k, 2m_k)&=2N -C(M_k)+2\Delta_k.    
\end{align}
\end{theorem}

\section{Proof of Theorem 1.1}
If there are infinitely many contractible minimal closed geodesics,
there is nothing to prove. Hence we make the following assumption.

{\bf (FCCG)} There exist only finitely many contractible minimal
closed geodesics $c_1,\ldots,c_q$ on $(S^{2n+1}/\Gamma,F)$, and
$i(c_j)\ge1$ for $1\le j\le q$.

After relabeling, let $c_1,\ldots,c_r$ be the homologically visible
ones. Proposition~\ref{Rad1992'} shows that, for $j>r$, the entire
iteration tower of $c_j$ has zero local critical modules. In
particular, invisible towers make no contribution to the Morse-type
numbers or to the resonance sum in Theorem 2.1. We put $\gamma_{c_j}=0$ for $j>r$.

Note that we have $\hat{i}(c_j)>0$ for all
$1\le j\le q$, otherwise there is a contractible minimal closed geodesic
$c$ on $M$ satisfying $\hat{i}(c)=0$. Then $i(c^m) =0$ for all $m\in \N$
by the Bott iteration formula, which contradicts $i(c_j)\geq 1$ for all
$1\le j\le q$. Define the finite positive integer $\bar{m}$ by
\begin{equation}
    \bar{m}=\max_{1\le j\le q}\left\{\min\{m_0\in\N\ |\ i(c_j^{m+\ell})\ge i(c_j^\ell), \quad
    \forall\ \ell\ge 1, m\geq m_0\}\right\}. \lb{4.1}
\end{equation}

Let $\overline M$ be the integer in (\ref{3.20}), and fix
\[
0<\epsilon<\frac{1}{1+2\overline M
\sum_{j=1}^{q}|\gamma_{c_j}|}.
\]
By Theorem 3.2, for the integer $\bar{m}$ defined in (\ref{4.1}),
we choose a $(q+1)$-tuple
$(N,m_1,\dots,m_q)\in\mathbb N^{q+1}$ such that
$n\mid N$, $N\ge2n$, (\ref{3.21}) holds for this $\epsilon$, and,
for all $1\le j\le q$, the following hold:
\bea
\bar{m}+2 &\le& \min\{2m_j,\ 1\le j\le q\},\lb{ECIJT1}\\
i(c_j^{2m_j-m}) &=& 2N-i(c_j^m),\quad 1\le m\le\bar{m}, \lb{ECIJT2}\\
i(c_j^{2m_j}) &=& 2N-C(P_j)+2\Delta_j,\lb{ECIJT3}\\
i(c_j^{2m_j+m})&=& 2N+i(c_j^m),\quad 1\le m\le\bar{m},\lb{ECIJT4}\eea
where $P_j=P_{c_j}\in\Sp(4n)$ is the linearized Poincar\'{e} map of $c_j$.
Note that $\nu(c_j^m) = 0$ for all $1 \le j \le q$ and $m \in \mathbb{N}$ due to the bumpy assumption.

To establish the uniform index bounds, we estimate $i(c_j^{2m_j \pm m})$ by dividing the iteration $m$ into two ranges:

{\bf{Case 1:}} $1 \le m \le \bar{m}$.

By the Bott-type formulae (cf. \cite{Bott1956} and Theorem 9.2.1 of \cite{lo2002}),
$i(c_j^m) \ge i(c_j)$ for any $m \ge 1$. Substituting this into (\ref{ECIJT2}) and (\ref{ECIJT4}) yields:
\bea
i(c_j^{2m_j-m}) &=& 2N-i(c_j^m)\leq 2N-i(c_j), \quad 1\le m\le\bar{m}, \lb{mbzbgj1}\\
i(c_j^{2m_j+m})&=& 2N+i(c_j^m)\geq 2N+i(c_j), \quad 1\le m\le\bar{m}. \lb{mbzbgj2}\eea
{\bf{Case 2:}} $m\ge\bar{m}+1$.

By the definition of $\bar{m}$ in (\ref{4.1}),
we can naturally bound the indices using the limits from the previous range:
\bea
i(c_j^{2m_j-m}) &\leq& i(c_j^{2m_j-1})= 2N-i(c_j),\lb{mbzbgj3}\\
i(c_j^{2m_j+m})&\geq& i(c_j^{2m_j+1})= 2N+i(c_j).\lb{mbzbgj4}\eea
Here the first inequality is used for $\bar{m}+1\le m<2m_j$,
whereas the second is valid for every $m\ge\bar{m}+1$.
Combining Case 1 and Case 2, we immediately obtain the desired estimates for all $1 \le j \le q$:
\bea
i(c_j^{2m_j-m}) &\leq& 2N-i(c_j),\quad \forall \ 1\leq m< 2m_j, \lb{zbgj1}\\
i(c_j^{2m_j+m})&\geq& 2N+i(c_j),\quad \forall \ m\geq 1.\lb{zbgj2}\eea

Before stating Lemma 4.1, we define the following notation:
\begin{align}
k_+^e &= {}^{\#}\{1 \leq j \leq q \mid i(c_j^{2m_j}) \geq k+1,\,
\dim\overline C_{i(c_j^{2m_j})}(E,c_j^{2m_j};[e])=1,\, i(c_j) \in 2\mathbb{Z}\},  \\
k_+^o &= {}^{\#}\{1 \leq j \leq q \mid i(c_j^{2m_j}) \geq k+1,\,
\dim\overline C_{i(c_j^{2m_j})}(E,c_j^{2m_j};[e])=1,\, i(c_j) \in 2\mathbb{Z}+1\}.
\end{align}

As noted in \cite{Rad1989}, \cite{Rad1992} and Section 2, the resonance identity and the Betti numbers of the loop space for spheres
coincide with those of quotient manifold. Therefore, we obtain the following lemma.

\begin{lemma}
    Under the assumption (FCCG), with the above $\bar{m}$,
    there exists a $(q+1)$-tuple
    \[
    (N, m_1, \dots, m_q) \in \mathbb{N}^{q+1}
    \]
    such that the following inequality holds:
    \begin{equation} \lb{jogj}
        (2N - 1)_+^e - (2N - 1)_+^o \geq n+2.
    \end{equation} 
\end{lemma}
{\bf{Proof}:} Define
\be
a_j(m)=(-1)^{i(c_j^m)}
\dim\overline C_{i(c_j^m)}(E,c_j^m;[e]). \lb{local-euler-contribution}
\ee
For a visible $c_j$, Proposition~\ref{Rad1992'} and (\ref{3.11})
give $a_j(m+2)=a_j(m)$. For an invisible $c_j$, the same identity
holds because $a_j(m)=0$ for every $m$. Consequently,
\begin{equation}
\begin{aligned}
\sum_{m=1}^{2m_j}a_j(m)
&=m_j\bigl(a_j(1)+a_j(2)\bigr)\\
&= 2m_j \gamma_{c_j}, \qquad \forall 1 \le j \le q,
\end{aligned} \lb{jch}
\end{equation}
where $\gamma_{c_j}$ is defined by (\ref{gamma-definition}) for a
visible geodesic and is zero for an invisible one.
Since $i(c_j^m)\ge i(c_j)\ge1$, we also have $M_0=0$.
Summing this over all $q$ contractible minimal closed geodesics,
and using the index bounds (\ref{zbgj1}) and (\ref{zbgj2}), we cut off the total sum at the index $2N - 1$ to get:
\begin{equation}
\sum_{j=1}^q\sum_{m=1}^{2m_j} (-1)^{i(c_j^m)} \dim \overline C_{i(c_j^m)}(E, c_j^m;[e]) = \sum_{i=1}^{2N-1} (-1)^i M_i + (2N-1)_+^e - (2N-1)_+^o.
\lb{4.16}
\end{equation}
We claim that the sum on the right-hand side of (\ref{jch}), taken
over $j$, is
\begin{equation}
\sum_{j=1}^q 2m_j \gamma_{c_j} = 2N \overline{B}(\Lambda_e M, \Lambda^0 M; \mathbb{Q}) = 2N \left( \frac{n+1}{2n} \right) = N \frac{n+1}{n}.
\lb{4.17}
\end{equation}
After relabeling, $c_1,\ldots,c_r$ are the visible geodesics and
$\gamma_{c_j}=0$ for $j>r$. Hence Theorem 2.1 gives the same
resonance sum with the upper limit $q$. Together with
(\ref{3.20})--(\ref{3.21}) of Theorem 3.2, this yields
\begin{equation}
\begin{aligned}
\left| 2N \overline{B}(\Lambda_e M, \Lambda^0 M; \mathbb{Q}) - \sum_{j=1}^q 2m_j \gamma_{c_j} \right|
&= \left| \sum_{j=1}^q \frac{2N \gamma_{c_j}}{\hat{i}(c_j)} - \sum_{j=1}^q 2\gamma_{c_j} \left( \left\lfloor \frac{N}{\overline M\hat{i}(c_j)} \right\rfloor + \chi_j \right) \overline M \right| \\
&\le 2\overline M \sum_{j=1}^q |\gamma_{c_j}| \left| \left\{ \frac{N}{\overline M\hat{i}(c_j)} \right\} - \chi_j \right| \\
&< 2\overline M \epsilon \sum_{j=1}^q |\gamma_{c_j}| \\
&< 1,
\end{aligned} \lb{4.18}
\end{equation}
by the above choice of $\epsilon$.
Since every $2m_j \gamma_{c_j}$ is an integer,
the identity (\ref{4.17}) holds exactly by (\ref{4.18}) and (\ref{reident1}).

On the other hand, Theorem 2.2 and Proposition 2.2 imply that
$\beta_i=0$ for all odd $i$ and $\beta_0=0$. Hence
\begin{equation} \label{betti_bound}
\begin{aligned}
\sum_{i=1}^{2N-1} (-1)^i M_i &\le \sum_{i=1}^{2N-1} (-1)^i \beta_i = \sum_{i=0}^{2N-1} \beta_i \\
&= \left\lfloor \frac{2N-1}{2n} \right\rfloor + \left\lfloor \frac{2N-1}{2} \right\rfloor - n \\ 
&= N\frac{n+1}{n} - (n+2),
\end{aligned}
\end{equation}
where the first equality holds since $\beta_i=0$ for all odd $i$ and $\beta_0=0$,
and the last equality holds since $n\mid N$ and $N\ge2n$. 
Now the inequality (\ref{jogj}) follows from (\ref{jch}), (\ref{4.16}), (\ref{4.17}) and (\ref{betti_bound}). \hfill $\square$

To prove Theorem 1.1, we need the following lemma.
\begin{lemma}(cf. Theorem 1.2 of \cite{Liu2019})
Let $M=S^{2n+1}/\Gamma$ be endowed with a bumpy Finsler metric $F$
carrying only finitely many contractible minimal closed geodesics
$c_1,\ldots,c_q$. If $i(c_j)\ge2$ for every $1\le j\le q$, then
$q\ge2n+2$.
\end{lemma}

{\it Proof.}
In Section 6 of \cite{Liu2019}, the curvature pinching is used to
obtain $i(c_j)\ge2$. Add the nonzero-local-module condition to the
four counts $N_\pm^{e/o}$ there; Proposition~\ref{Rad1992'} then
removes every invisible tower. The first ECIJT(Enhanced Common Index Jump Theorem) tuple, chosen with
$n\mid N$ and $N\ge2n$, and its complementary tuple give
$N_+^e\ge n$ and $N_-^e\ge n$, hence $2n$ distinct visible
even-index geodesics. The uniform index bounds show that these towers
do not contribute in degree $2N$ and that every remaining tower
contributes at most one dimension. Since
$M_{2N}\ge\beta_{2N}=2$, two further distinct geodesics are required,
and therefore $q\ge2n+2$.
\hfill $\square$

{\bf{Proof of Theorem 1.1}}: By Lemma 4.1 and the definitions of
$(2N-1)_+^e$ and $(2N-1)_+^o$, we have
\begin{equation*}
q\ge(2N-1)_+^e\ge(2N-1)_+^e-(2N-1)_+^o\ge n+2.
\end{equation*}
We prove the conclusion by contradiction. Assume that there exist only $n+2$ contractible minimal closed geodesics.
Then every inequality above is an equality. In particular,
\[
q=(2N-1)_+^e=n+2,\qquad (2N-1)_+^o=0.
\]
Thus every $c_j$ is counted by $(2N-1)_+^e$ and has positive even
Morse index. Hence
\begin{equation*}
    i(c_j)\ge 2, \ \forall j=1,2, \dots, n+2.
\end{equation*}
Lemma 4.2 now gives $q\ge2n+2>n+2$, a contradiction. This completes
the proof of Theorem 1.1.

\end{document}